\documentclass[a4paper,12pt]{article}
\usepackage[utf8]{inputenc}
\usepackage[T1]{fontenc}
\usepackage{amsfonts,amsmath,amssymb,amsthm}
\theoremstyle{definition}

\newtheorem{conjecture}{Conjecture}
\theoremstyle{plain}

\newtheorem{proposition}{Proposition}
\newtheorem{theorem}{Theorem}
\usepackage{graphicx}
\usepackage{caption}
\usepackage{pgf,tikz,circuitikz,subfig}
\usetikzlibrary{arrows,shapes,positioning}
\usepackage{booktabs,multirow,tabularx}
\usepackage{hyperref}
\hypersetup{colorlinks,%
	citecolor={red}, 
	urlcolor={blue},
	linkcolor={blue},
	breaklinks={true},
	pagebackref={true},
	hyperindex={true},
}
\usepackage{url}
\usepackage{cite}
\usepackage[left=2cm,right=2cm,top=2cm,bottom=2cm]{geometry}
\usepackage[ruled]{algorithm}
\usepackage{algorithmic}
\usepackage{newtxtext}
\usepackage{courier}
\usepackage{enumitem}
\newlist{abbrv}{itemize}{1}
\setlist[abbrv,1]{label=,labelwidth=0.9in,align=parleft,noitemsep,leftmargin=!}

\newcommand{\rv}[1]{\boldsymbol{#1}}

\newcommand{\ub}[1]{\overline{#1}}

\newcommand{\geo}[1]{\mathtt{#1}}

\DeclareMathOperator{\subj}{s.t.}

\title{Tight bounds on the maximal perimeter and the maximal width of convex small polygons}
\author{Christian Bingane\thanks{D\'{e}partement de math\'{e}matiques et de g\'{e}nie industriel, Polytechnique Montr\'{e}al, Montreal, QC, Canada. Email: \url{christian.bingane@polymtl.ca}}}
\begin{document}
\maketitle
\begin{abstract}
A small polygon is a polygon of unit diameter. The maximal perimeter and the maximal width of a convex small polygon with $n=2^s$ vertices are not known when $s \ge 4$. In this paper, we construct a family of convex small $n$-gons, $n=2^s$ and $s\ge 3$, and show that the perimeters and the widths obtained cannot be improved for large $n$ by more than $a/n^6$ and $b/n^4$ respectively, for certain positive constants $a$ and $b$. In addition, assuming that a conjecture of Mossinghoff is true, we formulate the maximal perimeter problem as a nonlinear optimization problem involving trigonometric functions and, for $n=2^s$ with $3 \le s\le 7$, we provide global optimal solutions.
\end{abstract}

\section{Introduction}
The {\em diameter} of a polygon is the largest Euclidean distance between pairs of its vertices. A polygon is said to be {\em small} if its diameter equals one. For a given integer $n \ge 3$, the maximal perimeter problem consists in finding a convex small $n$-gon with the longest perimeter. The problem was first investigated by Reinhardt~\cite{reinhardt1922} in 1922, and later by Datta~\cite{datta1997} in 1997. They proved that
\begin{itemize}
	\item for all $n \ge 3$, the value $2n\sin \frac{\pi}{2n}$ is an upper bound on the perimeter of a convex small $n$-gon;
	\item when $n$ is odd, the regular small $n$-gon is an optimal solution, but it is unique only if $n$ is prime;
	\item when $n$ is even, the regular small $n$-gon is not optimal;
	\item when $n$ has an odd factor, there are finitely many optimal solutions~\cite{mossinghoff2011,hare2013,hare2019} and they are all equilateral.
\end{itemize}

When $n$ is a power of $2$, the maximal perimeter problem is solved for $n \le 8$. In 1987, Tamvakis~\cite{tamvakis1987} found the unique convex small $4$-gon with the longest perimeter, shown in Figure~\ref{figure:4gon:R3+}. In 2007, Audet, Hansen, and Messine~\cite{audet2007a} used both geometrical arguments and methods of global optimization to determine the unique convex small $8$-gon with the longest perimeter, illustrated in Figure~\ref{figure:8gon:V8}.

The diameter graph of a small polygon is the graph with the vertices of the polygon, and an edge between two vertices exists only if the distance between these vertices equals one. Figure~\ref{figure:4gon}, Figure~\ref{figure:6gon}, and Figure~\ref{figure:8gon} represent diameter graphs of some convex small polygons. The solid lines illustrate pairs of vertices which are unit distance apart. Mossinghoff~\cite{mossinghoff2006b} conjectured that, for $n\ge 4$ power of $2$, the diameter graph of a convex small $n$-gon with maximal perimeter has a cycle of length $n/2+1$, plus $n/2-1$ additional pendant edges, arranged so that all but two particular vertices of the cycle have a pendant edge. For example, Figure~\ref{figure:4gon:R3+} and Figure~\ref{figure:8gon:V8} exhibit the diameter graphs of optimal $n$-gons when $n=4$ and when $n=8$ respectively. We point out that numerical values in figures and tables in this paper are rounded at the last reported digit.

\begin{figure}
	\centering
	\subfloat[$(\geo{R}_4,2.8284,0.7071)$]{
		\begin{tikzpicture}[scale=4]
			\draw[dashed] (0,0) -- (0.5000,0.5000) -- (0,1) -- (-0.5000,0.5000) -- cycle;
			\draw (0,0) -- (0,1);
			\draw (0.5000,0.5000) -- (-0.5000,0.5000);
		\end{tikzpicture}
	}
	\subfloat[$(\geo{R}_3^+,3.0353,0.8660)$]{
		\begin{tikzpicture}[scale=4]
			\draw[dashed] (0.5000,0.8660) -- (0,1) -- (-0.5000,0.8660);
			\draw (0,1) -- (0,0) -- (0.5000,0.8660) -- (-0.5000,0.8660) -- (0,0);
		\end{tikzpicture}
		\label{figure:4gon:R3+}
	}
	\caption{Two convex small $4$-gons $(\geo{P}_4,L(\geo{P}_4)),W(\geo{P}_4))$}
	\label{figure:4gon}
\end{figure}
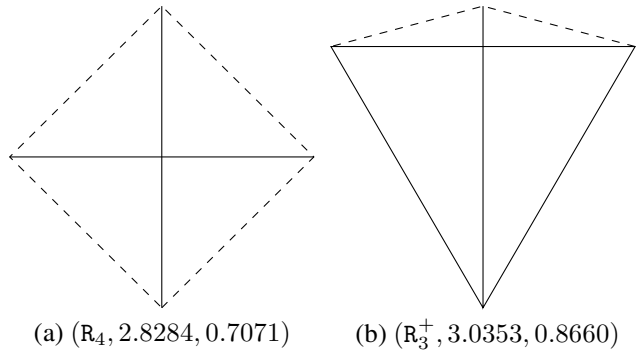

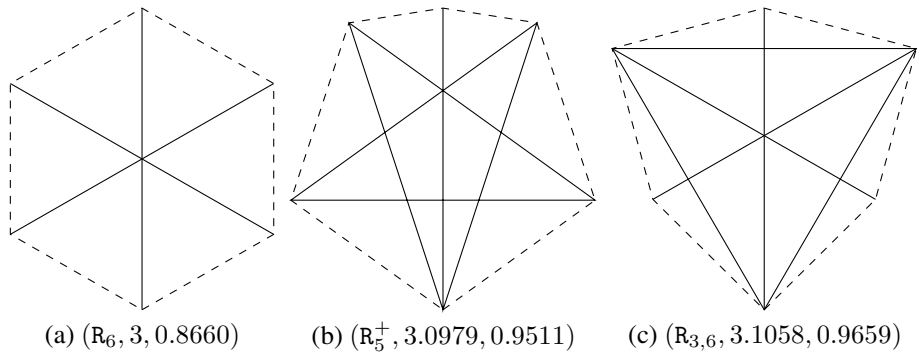
\begin{figure}
	\centering
	\subfloat[$(\geo{R}_6,3,0.8660)$]{
		\begin{tikzpicture}[scale=4]
			\draw[dashed] (0,0) -- (0.4330,0.2500) -- (0.4330,0.7500) -- (0,1) -- (-0.4330,0.7500) -- (-0.4330,0.2500) -- cycle;
			\draw (0,0) -- (0,1);
			\draw (0.4330,0.2500) -- (-0.4330,0.7500);
			\draw (0.4330,0.7500) -- (-0.4330,0.2500);
		\end{tikzpicture}
	}
	\subfloat[$(\geo{R}_5^+,3.0979,0.9511)$]{
		\begin{tikzpicture}[scale=4]
			\draw[dashed] (0,0) -- (0.5000,0.3633) -- (0.3090,0.9511) -- (0,1) -- (-0.3090,0.9511) -- (-0.5000,0.3633) -- cycle;
			\draw (0,1) -- (0,0) -- (0.3090,0.9511) -- (-0.5000,0.3633) -- (0.5000,0.3633) -- (-0.3090,0.9511) -- (0,0);
		\end{tikzpicture}
	}
	\subfloat[$(\geo{R}_{3,6},3.1058,0.9659)$]{
		\begin{tikzpicture}[scale=4]
			\draw[dashed] (0,0) -- (0.3660,0.3660) -- (0.5000,0.8660) -- (0,1) -- (-0.5000,0.8660) -- (-0.3660,0.3660) -- cycle;
			\draw (0,0) -- (0.5000,0.8660) -- (-0.5000,0.8660) -- cycle;
			\draw (0,0) -- (0,1);
			\draw (0.3660,0.3660) -- (-0.5000,0.8660);
			\draw (0.5000,0.8660) -- (-0.3660,0.3660);
		\end{tikzpicture}
		\label{figure:6gon:R36}
	}
	\caption{Three convex small $6$-gons $(\geo{P}_6,L(\geo{P}_6)),W(\geo{P}_6))$}
	\label{figure:6gon}
\end{figure}

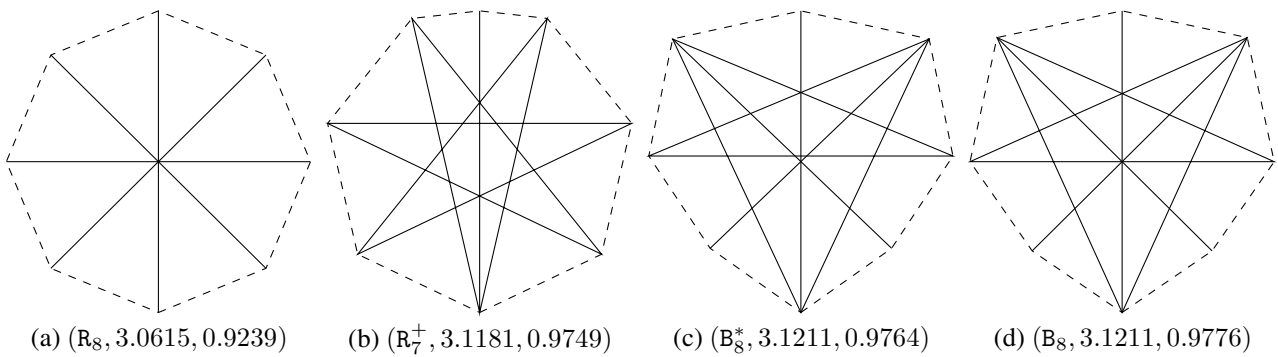
\begin{figure}
	\centering
	\subfloat[$(\geo{R}_8,3.0615,0.9239)$]{
		\begin{tikzpicture}[scale=4]
			\draw[dashed] (0,0) -- (0.3536,0.1464) -- (0.5000,0.5000) -- (0.3536,0.8536) -- (0,1) -- (-0.3536,0.8536) -- (-0.5000,0.5000) -- (-0.3536,0.1464) -- cycle;
			\draw (0,0) -- (0,1);
			\draw (0.3536,0.1464) -- (-0.3536,0.8536);
			\draw (0.5000,0.5000) -- (-0.5000,0.5000);
			\draw (0.3536,0.8536) -- (-0.3536,0.1464);
		\end{tikzpicture}
	}
	\subfloat[$(\geo{R}_7^+,3.1181,0.9749)$]{
		\begin{tikzpicture}[scale=4]
			\draw[dashed] (0,0) -- (0.4010,0.1931) -- (0.5000,0.6270) -- (0.2225,0.9749) -- (0,1) -- (-0.2225,0.9749) -- (-0.5000,0.6270) -- (-0.4010,0.1931) -- cycle;
			\draw (0,1) -- (0,0) -- (0.2225,0.9749) -- (-0.4010,0.1931) -- (0.5000,0.6270) -- (-0.5000,0.6270) -- (0.4010,0.1931) -- (-0.2225,0.9749) -- (0,0);
		\end{tikzpicture}
	}
	\subfloat[$(\geo{B}_8^*,3.1211,0.9764)$]{
		\begin{tikzpicture}[scale=4]
			\draw[dashed] (0,0) -- (0.2983,0.2128) -- (0.5000,0.5188) -- (0.4217,0.9067) -- (0,1) -- (-0.4217,0.9067) -- (-0.5000,0.5188) -- (-0.2983,0.2128) -- cycle;
			\draw (0,0) -- (0,1);
			\draw (0,0) -- (0.4217,0.9067) -- (-0.5000,0.5188) -- (0.5000,0.5188)-- (-0.4217,0.9067) -- cycle;
			\draw (0.4217,0.9067) -- (-0.2983,0.2128);\draw (-0.4217,0.9067) -- (0.2983,0.2128);
		\end{tikzpicture}
		\label{figure:8gon:V8}
	}
	\subfloat[$(\geo{B}_8,3.1211,0.9776)$]{
		\begin{tikzpicture}[scale=4]
			\draw[dashed] (0,0) -- (0.2957,0.2043) -- (0.5000,0.5000) -- (0.4114,0.9114) -- (0,1) -- (-0.4114,0.9114) -- (-0.5000,0.5000) -- (-0.2957,0.2043) -- cycle;
			\draw (0,0) -- (0.4114,0.9114) -- (-0.5000,0.5000) -- (0.5000,0.5000) -- (-0.4114,0.9114) -- cycle;
			\draw (0,0) -- (0,1);
			\draw (0.4114,0.9114) -- (-0.2957,0.2043);\draw (-0.4114,0.9114) -- (0.2957,0.2043);
		\end{tikzpicture}
		\label{figure:8gon:B8}
	}
	\caption{Four convex small $8$-gons $(\geo{P}_8,L(\geo{P}_8)),W(\geo{P}_8))$}
	\label{figure:8gon}
\end{figure}

The {\em width}­ of a polygon in some direction is the distance between two parallel lines perpendicular to this direction and supporting the polygon from below and above. The width of a polygon is the minimum width for all directions. For a given integer $n \ge 3$, the maximal width problem consists in finding a convex small $n$-gon with the largest width. This problem was partially solved by Bezdek and Fodor~\cite{bezdek2000} in~2000. They proved that
\begin{itemize}
	\item for all $n \ge 3$, the value $\cos \frac{\pi}{2n}$ is an upper bound on the width of a convex small $n$-gon;
	\item when $n$ has an odd factor, a convex small $n$-gon is optimal for the maximal width problem if and only if it is optimal for the maximal perimeter problem;
	\item when $n = 4$, there are infinitely many optimal convex small $4$-gons, including the $4$-gon illustrated in Figure~\ref{figure:4gon:R3+}.
\end{itemize}

When $n\ge 8$ is a power of $2$, the maximal width is only known for the first open case $n = 8$. In 2013, Audet, Hansen, Messine, and Ninin~\cite{audet2013} combined geometrical and analytical reasoning as well as methods of global optimization to prove that there are infinitely many optimal convex small $8$-gons, including the $8$-gon illustrated in Figure~\ref{figure:8gon:B8}.

For $n=2^s$ with integer $s\ge 4$, exact solutions in both problems appear to be presently out of reach. However, tight lower bounds on the maximal perimeter and the maximal width may be obtained analytically. For instance, Mossinghoff~\cite{mossinghoff2006b} constructed convex small $n$-gons, for $n=2^s$ with $s\ge 3$, and proved that the perimeters obtained cannot be improved for large~$n$ by more than~$\pi^5/(16n^5)$. We can also show that, when $n=2^s$ with $s\ge 2$, the value $\cos \frac{\pi}{2n-2}$ is a lower bound on the maximal width and this bound cannot be improved for large~$n$ by more than~$\pi^2/(4n^3)$. In this paper, we propose tighter lower bounds on both the maximal perimeter and the maximal width of convex small $n$-gons when $n=2^s$ and integer $s \ge 3$. Thus, the main result of this paper is the following:

\begin{theorem}\label{thm:Bn}
	For a given integer $n \ge 3$, let $\ub{L}_n := 2n \sin \frac{\pi}{2n}$ denote an upper bound on the perimeter $L(\geo{P}_n)$ of a convex small $n$-gon $\geo{P}_n$, and $\ub{W}_n := \cos \frac{\pi}{2n}$ denote an upper bound on its width $W(\geo{P}_n)$. If $n = 2^s$ with $s\ge 3$, then there exists a convex small $n$-gon $\geo{B}_n$ such that
	\[
	\begin{aligned}
		L(\geo{B}_n) &= 2n \sin \frac{\pi}{2n} \cos \left(\frac{\pi}{2n}-\frac{1}{2}\arcsin\left(\frac{1}{2}\sin \frac{2\pi}{n}\right)\right),\\
		W(\geo{B}_n) &= \cos \left(\frac{\pi}{n} - \frac{1}{2}\arcsin\left(\frac{1}{2}\sin \frac{2\pi}{n}\right)\right),
	\end{aligned}
	\]
	and
	\[
	\begin{aligned}
		\ub{L}_n - L(\geo{B}_n) &= \frac{\pi^7}{32n^6} + O\left(\frac{1}{n^8}\right),\\
		\ub{W}_n - W(\geo{B}_n) &= \frac{\pi^4}{8n^4} + O\left(\frac{1}{n^6}\right).
	\end{aligned}
	\]
\end{theorem}

The remainder of this paper is organized as follows. Section~\ref{sec:ngon} recalls principal results on the maximal perimeter and the maximal width of convex small polygons. We prove Theorem~\ref{thm:Bn} in Section~\ref{sec:Bn}. Tight bounds on the maximal width of unit-perimeter $n$-gons, $n=2^s$ and $s\ge 3$, are deduced from Theorem~\ref{thm:Bn} in Section~\ref{sec:wn}. Under the assumption that Mossinghoff's conjecture is true, a nonlinear optimization problem involving trigonometric functions is proposed for the maximal perimeter problem in Section~\ref{sec:nlo}. Global optimal solutions obtained by using AMPL with the solver Couenne~\cite{belotti2009} are given for $n= 2^s$ with $3\le s \le 7$. Section~\ref{sec:conclusion} concludes the paper.

\section{Perimeters and widths of convex small polygons}\label{sec:ngon}
\subsection{Maximal perimeter and maximal width}
Let $L(\geo{P})$ denote the perimeter of a polygon $\geo{P}$ and $W(\geo{P})$ its width. For a given integer $n\ge 3$, let $\geo{R}_n$ denote the regular small $n$-gon. We have
\[
L(\geo{R}_n) =
\begin{cases}
	2n\sin \frac{\pi}{2n} &\text{if $n$ is odd,}\\
	n\sin \frac{\pi}{n} &\text{if $n$ is even,}\\
\end{cases}
\]
and
\[
W(\geo{R}_n) =
\begin{cases}
	\cos \frac{\pi}{2n} &\text{if $n$ is odd,}\\
	\cos \frac{\pi}{n} &\text{if $n$ is even.}\\
\end{cases}
\]
We remark that $L(\geo{R}_n) < L(\geo{R}_{n-1})$~\cite{audet2009a} and $W(\geo{R}_n) < W(\geo{R}_{n-1})$ for all even $n\ge 4$. The polygon $\geo{R}_n$ does not have maximum perimeter nor maximum width for any even $n\ge 4$. Indeed, when $n$ is even, one can construct a convex small $n$-gon with a longer perimeter and a larger width than $\geo{R}_n$ by adding a vertex at distance $1$ along the mediatrix of an angle in $\geo{R}_{n-1}$. We denote this $n$-gon by $\geo{R}_{n-1}^+$ and we have
\[
\begin{aligned}
	L(\geo{R}_{n-1}^+) &= (2n-2)\sin \frac{\pi}{2n-2} + 4 \sin \frac{\pi}{4n-4} - 2\sin \frac{\pi}{2n-2},\\
	W(\geo{R}_{n-1}^+) &= \cos \frac{\pi}{2n-2}.
\end{aligned}
\]

When $n$ has an odd factor $m$, we construct another family of convex equilateral small $n$-gons as follows:
\begin{enumerate}
	\item Consider a regular small $m$-gon $\geo{R}_m$;
	\item Transform $\geo{R}_m$ into a Reuleaux $m$-gon by replacing each edge by a circle's arc passing through its end vertices and centered at the opposite vertex;
	\item Add at regular intervals $n/m-1$ vertices within each arc;
	\item Take the convex hull of all vertices.
\end{enumerate}
We denote these $n$-gons by $\geo{R}_{m,n}$ and we have
\[
\begin{aligned}
	L(\geo{R}_{m,n}) &= 2n\sin \frac{\pi}{2n},\\
	W(\geo{R}_{m,n}) &= \cos \frac{\pi}{2n}.
\end{aligned}
\]
The $6$-gon $\geo{R}_{3,6}$ is illustrated in Figure~\ref{figure:6gon:R36}.

\begin{theorem}[Reinhardt~\cite{reinhardt1922}, Datta~\cite{datta1997}]\label{thm:perimeter}
	For all $n \ge 3$, let $L_n^*$ denote the maximal perimeter among all convex small $n$-gons and let $\ub{L}_n := 2n \sin \frac{\pi}{2n}$.
	\begin{itemize}
		\item When $n$ has an odd factor $m$, $L_n^* = \ub{L}_n$ is achieved by finitely many equilateral $n$-gons~\cite{mossinghoff2011,hare2013,hare2019}, including~$\geo{R}_{m,n}$. The optimal $n$-gon $\geo{R}_{m,n}$ is unique if $m$ is prime and $n/m \le 2$.
		\item When $n=2^s$ with integer $s\ge 2$, $L(\geo{R}_n) < L_n^* < \ub{L}_n$.
	\end{itemize}
\end{theorem}

When $n=2^s$, the maximal perimeter $L_n^*$ is only known for $s \le 3$. Tamvakis~\cite{tamvakis1987} found that $L_4^* = 2+\sqrt{6}-\sqrt{2}$, and this value is achieved only by $\geo{R}_3^+$, shown in Figure~\ref{figure:4gon:R3+}. Audet, Hansen, and Messine~\cite{audet2007a} found that $L_8^* \approx 3.121147$, and this value is only achieved by $\geo{B}_8^*$, shown in Figure~\ref{figure:8gon:V8}.

\begin{theorem}[Bezdek and Fodor~\cite{bezdek2000}]\label{thm:width}
	For all $n \ge 3$, let $W_n^*$ denote the maximal width among all convex small $n$-gons and let $\ub{W}_n := \cos \frac{\pi}{2n}$.
	\begin{itemize}
		\item When $n$ has an odd factor, $W_n^* = \ub{W}_n$ is achieved by a convex small $n$-gon with maximal perimeter $L_n^* = \ub{L}_n$.
		\item When $n=2^s$ with integer $s\ge 2$, $W(\geo{R}_n) < W_n^* < \ub{W}_n$.
	\end{itemize}
\end{theorem}

When $n = 2^s$, the maximal width $W_n^*$ is only known for $s \le 3$. Bezdek and Fodor~\cite{bezdek2000} showed that $W_4^* = \frac{1}{2}\sqrt{3}$, and this value is achieved by infinitely many convex small $4$-gons, including~$\geo{R}_3^+$ shown in Figure~\ref{figure:4gon:R3+}. Audet, Hansen, Messine, and Ninin found that $W_8^* = \frac{1}{4}\sqrt{10+2\sqrt{7}}$, and this value is also achieved by infinitely many convex small $8$-gons, including~$\geo{B}_8$ shown in Figure~\ref{figure:8gon:B8}. It is interesting to note that while the optimal $4$-gon for the maximal perimeter problem is also optimal for the maximal width problem, the optimal $8$-gon for the maximal perimeter problem is not optimal for the maximal width problem.

\subsection{Lower bounds on the maximal perimeter and the maximal width}
For $n=2^s$ with integer $s\ge 2$, let $\geo{T}_n$ denote the convex $n$-gon obtained by subdividing each bounding arc of a such Reuleaux triangle into either $\lceil n/3 \rceil$ or $\lfloor n/3 \rfloor$ subarcs of equal length, then taking the convex hull of the endpoints of these arcs. For a real number $a$, $\lceil a \rceil$ is the least integer greater than or equal to $a$, and $\lfloor a \rfloor$ is the greatest integer less than or equal to $a$. We illustrate $\geo{T}_n$ for some $n$ in Figure~\ref{figure:Tn}. For each $n$, the perimeter of~$\geo{T}_n$ is given by
\[
L(\geo{T}_n) =
\begin{cases}
	\frac{4n-4}{3}\sin \frac{\pi}{2n-2} + \frac{2n+4}{3}\sin \frac{\pi}{2n+4} &\text{if $n = 3k+1$,}\\
	\frac{4n+4}{3}\sin \frac{\pi}{2n+2} + \frac{2n-4}{3}\sin \frac{\pi}{2n-4} &\text{if $n = 3k+2$.}\\
\end{cases}
\]

\begin{figure}
	\centering
	\subfloat[$(\geo{T}_8,3.1191,0.9659)$]{
		\begin{tikzpicture}[scale=4]
			\draw[dashed] (0,0) -- (0.2660,0.2232) -- (0.4397,0.5240) -- (0.5000,0.8660) -- (0,1) -- (-0.5000,0.8660) -- (-0.4397,0.5240) -- (-0.2660,0.2232) -- cycle;
			\draw[blue,thick] (0,0) -- (0.5000,0.8660) -- (-0.5000,0.8660) -- cycle;
			\draw[red,thick] (0,0) -- (0,1);
			\draw (0.5000,0.8660) -- (-0.4397,0.5240);\draw (0.5000,0.8660) -- (-0.2660,0.2232);
			\draw (-0.5000,0.8660) -- (0.4397,0.5240);\draw (-0.5000,0.8660) -- (0.2660,0.2232);
		\end{tikzpicture}
	}
	\subfloat[$(\geo{T}_{16},3.1364,0.9945)$]{
		\begin{tikzpicture}[scale=4]
			\draw[dashed] (0,0) -- (0.1691,0.1229) -- (0.3090,0.2782) -- (0.4135,0.4593) -- (0.4781,0.6581) -- (0.5000,0.8660) -- (0.3420,0.9397) -- (0.1736,0.9848) -- (0,1) -- (-0.1736,0.9848) -- (-0.3420,0.9397) -- (-0.5000,0.8660) -- (-0.4781,0.6581) -- (-0.4135,0.4593) -- (-0.3090,0.2782) -- (-0.1691,0.1229) -- cycle;
			\draw[blue,thick] (0,0) -- (0.5000,0.8660) -- (-0.5000,0.8660) -- cycle;
			\draw[red,thick] (0,0) -- (0,1);
			\draw (0,0) -- (0.3420,0.9397);\draw (0,0) -- (-0.3420,0.9397);
			\draw (0,0) -- (0.1736,0.9848);\draw (0,0) -- (-0.1736,0.9848);
			\draw (0.5000,0.8660) -- (-0.4781,0.6581);\draw (-0.5000,0.8660) -- (0.4781,0.6581);
			\draw (0.5000,0.8660) -- (-0.4135,0.4593);\draw (-0.5000,0.8660) -- (0.4135,0.4593);
			\draw (0.5000,0.8660) -- (-0.3090,0.2782);\draw (-0.5000,0.8660) -- (0.3090,0.2782);
			\draw (0.5000,0.8660) -- (-0.1691,0.1229);\draw (-0.5000,0.8660) -- (0.1691,0.1229);
		\end{tikzpicture}
	}
	\subfloat[$(\geo{T}_{32},3.1403,0.9986)$]{
		\begin{tikzpicture}[scale=4]
			\draw[dashed] (0,0) -- (0.0801,0.0514) -- (0.1549,0.1103) -- (0.2237,0.1759) -- (0.2861,0.2479) -- (0.3413,0.3254) -- (0.3888,0.4078) -- (0.4284,0.4944) -- (0.4595,0.5843) -- (0.4819,0.6768) -- (0.4955,0.7710) -- (0.5000,0.8660) -- (0.4067,0.9135) -- (0.3090,0.9511) -- (0.2079,0.9781) -- (0.1045,0.9945) -- (0,1) -- (-0.1045,0.9945) -- (-0.2079,0.9781) -- (-0.3090,0.9511) -- (-0.4067,0.9135) -- (-0.5000,0.8660) -- (-0.4955,0.7710) -- (-0.4819,0.6768) -- (-0.4595,0.5843) -- (-0.4284,0.4944) -- (-0.3888,0.4078) -- (-0.3413,0.3254) -- (-0.2861,0.2479) -- (-0.2237,0.1759) -- (-0.1549,0.1103) -- (-0.0801,0.0514) -- cycle;
			\draw[blue,thick] (0,0) -- (0.5000,0.8660) -- (-0.5000,0.8660) -- cycle;
			\draw[red,thick] (0,0) -- (0,1);
			\draw (0,0) -- (0.1045,0.9945);\draw (0,0) -- (-0.1045,0.9945);
			\draw (0,0) -- (0.2079,0.9781);\draw (0,0) -- (-0.2079,0.9781);
			\draw (0,0) -- (0.3090,0.9511);\draw (0,0) -- (-0.3090,0.9511);
			\draw (0,0) -- (0.4067,0.9135);\draw (0,0) -- (-0.4067,0.9135);
			\draw (0.5000,0.8660) -- (-0.4955,0.7710);\draw (-0.5000,0.8660) -- (0.4955,0.7710);
			\draw (0.5000,0.8660) -- (-0.4819,0.6768);\draw (-0.5000,0.8660) -- (0.4819,0.6768);
			\draw (0.5000,0.8660) -- (-0.4595,0.5843);\draw (-0.5000,0.8660) -- (0.4595,0.5843);
			\draw (0.5000,0.8660) -- (-0.4284,0.4944);\draw (-0.5000,0.8660) -- (0.4284,0.4944);
			\draw (0.5000,0.8660) -- (-0.3888,0.4078);\draw (-0.5000,0.8660) -- (0.3888,0.4078);
			\draw (0.5000,0.8660) -- (-0.3413,0.3254);\draw (-0.5000,0.8660) -- (0.3413,0.3254);
			\draw (0.5000,0.8660) -- (-0.2861,0.2479);\draw (-0.5000,0.8660) -- (0.2861,0.2479);
			\draw (0.5000,0.8660) -- (-0.2237,0.1759);\draw (-0.5000,0.8660) -- (0.2237,0.1759);
			\draw (0.5000,0.8660) -- (-0.1549,0.1103);\draw (-0.5000,0.8660) -- (0.1549,0.1103);
			\draw (0.5000,0.8660) -- (-0.0801,0.0514);\draw (-0.5000,0.8660) -- (0.0801,0.0514);
		\end{tikzpicture}
	}
	\caption{Tamvakis polygons $(\geo{T}_n,L(\geo{T}_n),W(\geo{T}_n))$}
	\label{figure:Tn}
\end{figure}
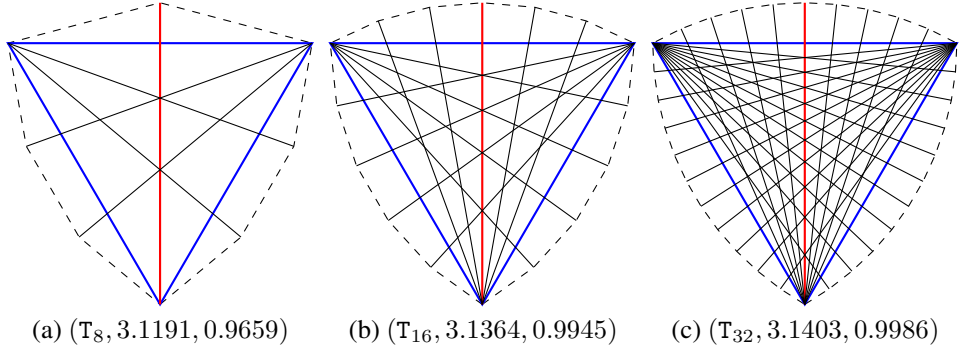

We note that $\geo{T}_4$ is optimal for the maximal perimeter problem and we can show that
\[
\ub{L}_n - L(\geo{T}_n) = \frac{\pi^3}{4n^4} + O \left(\frac{1}{n^5}\right)
\]
for all $n=2^s$ and $s\ge 2$. By contrast,
\[
\begin{aligned}
	\ub{L}_n - L(\geo{R}_n) &= \frac{\pi^3}{8n^2} + O \left(\frac{1}{n^4}\right),\\
	\ub{L}_n - L(\geo{R}_{n-1}^+) &= \frac{5\pi^3}{96n^3} + O \left(\frac{1}{n^4}\right)
\end{aligned}
\]
for all even $n\ge 4$. Tamvakis asked if $\geo{T}_n$ is also optimal when $s \ge 3$. Obviously, $\geo{T}_8$ is not optimal, i.e., $L(\geo{T}_8) < L_8^*$.

For all $n = 2^s$ with integer $s\ge 2$, let $\geo{P}_n^*$ denote a convex small $n$-gon with the longest perimeter.
\begin{conjecture}[Mossinghoff~\cite{mossinghoff2006b}]
	\label{thm:perimeter:diam}
	For all $n = 2^s$ with integer $s\ge 2$, the diameter graph of $\geo{P}_n^*$ has a cycle of length $n/2+1$, plus $n/2-1$ additional pendant edges, arranged so that all but two particular vertices of the cycle have a pendant edge.
\end{conjecture}
\begin{conjecture}[Mossinghoff~\cite{mossinghoff2006b}]
	\label{thm:perimeter:sym}
	For all $n = 2^s$ with integer $s\ge 2$, $\geo{P}_n^*$ has an axis of symmetry corresponding to one particular pendant edge in its diameter graph.
\end{conjecture}

Conjecture~\ref{thm:perimeter:diam} is proven for $n=4$~\cite{tamvakis1987} and $n=8$~\cite{audet2007a}. Conjecture~\ref{thm:perimeter:sym} is only proven for $n=4$~\cite{tamvakis1987}, but it is shown numerically for $n=8$ in~\cite{audet2007a}. Mossinghoff~\cite{mossinghoff2006b} constructed a family of convex small $n$-gons~$\geo{M}_n$ having the diameter graph described in Conjectures~\ref{thm:perimeter:diam} and~\ref{thm:perimeter:sym}. These polygons have the property that
	\[
	\ub{L}_n - L(\geo{M}_n) = \frac{\pi^5}{16n^5} + O \left(\frac{1}{n^6}\right)
	\]
	when $n=2^s$ and $s\ge 3$. We show $\geo{M}_n$ for some $n$ in Figure~\ref{figure:Mn}.

\begin{figure}
	\centering
	\subfloat[$(\geo{M}_8,3.1210,0.9747)$]{
		\begin{tikzpicture}[scale=4]
			\draw[dashed] (0,0) -- (0.3025,0.2254) -- (0.5000,0.5467) -- (0.4354,0.9002) -- (0,1) -- (-0.4354,0.9002) -- (-0.5000,0.5467) -- (-0.3025,0.2254) -- cycle;
			\draw[red,thick] (0,0) -- (0,1);
			\draw[blue,thick] (0,0) -- (0.4354,0.9002) -- (-0.5000,0.5467) -- (0.5000,0.5467) -- (-0.4354,0.9002) -- cycle;
			\draw (0.4354,0.9002) -- (-0.3025,0.2254);\draw (-0.4354,0.9002) -- (0.3025,0.2254);
		\end{tikzpicture}
	}
	\subfloat[$(\geo{M}_{16},3.1365,0.9943)$]{
		\begin{tikzpicture}[scale=4]
			\draw[dashed] (0,0) -- (0.1844,0.0591) -- (0.3539,0.1526) -- (0.4445,0.3246) -- (0.5000,0.5108) -- (0.4815,0.7022) -- (0.3599,0.8529) -- (0.2116,0.9774) -- (0,1) -- (-0.2116,0.9774) -- (-0.3599,0.8529) -- (-0.4815,0.7022) -- (-0.5000,0.5108) -- (-0.4445,0.3246) -- (-0.3539,0.1526) -- (-0.1844,0.0591) -- cycle;
			\draw[red,thick] (0,0) -- (0,1);
			\draw[blue,thick] (0,0) -- (0.2116,0.9774) -- (-0.3539,0.1526) -- (0.4815,0.7022) -- (-0.5000,0.5108) -- (0.5000,0.5108) -- (-0.4815,0.7022) -- (0.3539,0.1526) -- (-0.2116,0.9774) -- cycle;
			\draw (0.2116,0.9774) -- (-0.1844,0.0591);\draw (-0.2116,0.9774) -- (0.1844,0.0591);
			\draw (-0.3539,0.1526) -- (0.3599,0.8529);\draw (0.3539,0.1526) -- (-0.3599,0.8529);
			\draw (0.4815,0.7022) -- (-0.4445,0.3246);\draw (-0.4815,0.7022) -- (0.4445,0.3246);
		\end{tikzpicture}
	}
	\subfloat[$(\geo{M}_{32},3.1403,0.9987)$]{
		\begin{tikzpicture}[scale=4]
			\draw[dashed] (0,0) -- (0.0967,0.0148) -- (0.1915,0.0389) -- (0.2753,0.0895) -- (0.3537,0.1480) -- (0.4118,0.2267) -- (0.4620,0.3107) -- (0.4857,0.4057) -- (0.5000,0.5026) -- (0.4952,0.6001) -- (0.4624,0.6923) -- (0.4207,0.7808) -- (0.3551,0.8534) -- (0.2828,0.9193) -- (0.1945,0.9614) -- (0.1025,0.9947) -- (0,1) -- (-0.1025,0.9947) -- (-0.1945,0.9614) -- (-0.2828,0.9193) -- (-0.3551,0.8534) -- (-0.4207,0.7808) -- (-0.4624,0.6923) -- (-0.4952,0.6001) -- (-0.5000,0.5026) -- (-0.4857,0.4057) -- (-0.4620,0.3107) -- (-0.4118,0.2267) -- (-0.3537,0.1480) -- (-0.2753,0.0895) -- (-0.1915,0.0389) -- (-0.0967,0.0148) -- cycle;
			\draw[red,thick] (0,0) -- (0,1);
			\draw[blue,thick] (0,0) -- (0.1025,0.9947) -- (-0.1915,0.0389) -- (0.2828,0.9193) -- (-0.3537,0.1480) -- (0.4207,0.7808) -- (-0.4620,0.3107) -- (0.4952,0.6001) -- (-0.5000,0.5026) -- (0.5000,0.5026) -- (-0.4952,0.6001) -- (0.4620,0.3107) -- (-0.4207,0.7808) -- (0.3537,0.1480) -- (-0.2828,0.9193) -- (0.1915,0.0389) -- (-0.1025,0.9947) -- cycle;
			\draw (0.1025,0.9947) -- (-0.0967,0.0148);\draw (-0.1025,0.9947) -- (0.0967,0.0148);
			\draw (-0.1915,0.0389) -- (0.1945,0.9614);\draw (0.1915,0.0389) -- (-0.1945,0.9614);
			\draw (0.2828,0.9193) -- (-0.2753,0.0895);\draw (-0.2828,0.9193) -- (0.2753,0.0895);
			\draw (-0.3537,0.1480) -- (0.3551,0.8534);\draw (0.3537,0.1480) -- (-0.3551,0.8534);
			\draw (0.4207,0.7808) -- (-0.4118,0.2267);\draw (-0.4207,0.7808) -- (0.4118,0.2267);
			\draw (-0.4620,0.3107) -- (0.4624,0.6923);\draw (0.4620,0.3107) -- (-0.4624,0.6923);
			\draw (0.4952,0.6001) -- (-0.4857,0.4057);\draw (-0.4952,0.6001) -- (0.4857,0.4057);
		\end{tikzpicture}
	}
	\caption{Mossinghoff polygons $(\geo{M}_n,L(\geo{M}_n),W(\geo{M}_n))$}
	\label{figure:Mn}
\end{figure}
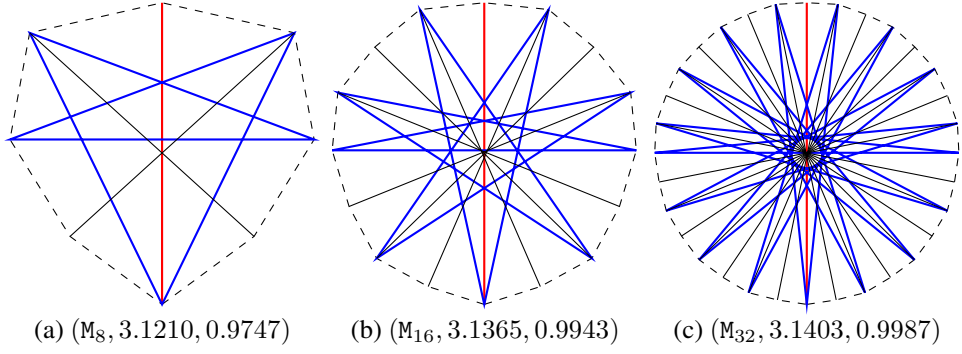

On the other hand, for all $n=2^s$ and integer $s\ge 3$,
\[
\begin{aligned}
	W(\geo{T}_n) &=
	\begin{cases}
		\cos \frac{\pi}{2n-2} &\text{if $n = 3k+1$,}\\
		\cos \frac{\pi}{2n-4} &\text{if $n = 3k+2$,}
	\end{cases}\\
	W(\geo{M}_n) &= \cos \left(\frac{\pi}{2n}+\frac{\pi^2}{4n^2} - \frac{\pi^2}{2n^3}\right),
\end{aligned}
\]
and we can show that $W(\geo{R}_{n-1}^+)\ge \max \{W(\geo{T}_n), W(\geo{M}_n)\}$. Note that
\[
\begin{aligned}
	\ub{W}_n - W(\geo{R}_n) &= \frac{3\pi^2}{8n^2} + O\left(\frac{1}{n^4}\right),\\
	\ub{W}_n - W(\geo{R}_{n-1}^+) &= \frac{\pi^2}{4n^3} + O\left(\frac{1}{n^4}\right)
\end{aligned}
\]
for all even $n\ge 4$.

\section{Proof of Theorem~\ref{thm:Bn}}\label{sec:Bn}
We use cartesian coordinates to describe an $n$-gon $\geo{P}_n$, assuming that a vertex $\geo{v}_i$, $i=0,1,\ldots,n-1$, is positioned at abscissa $x_i$ and ordinate $y_i$. Placing the vertex $\geo{v}_0$ at the origin, we set $x_0 = y_0 = 0$. We also assume that the $n$-gon $\geo{P}_n$ is in the half-plane $y\ge 0$.

For all $n=2^s$ with integer $s\ge 3$, consider the $n$-gon $\geo{P}_n$ having an $(n/2+1)$-length cycle: $\geo{v}_{0}-\geo{v}_1-\ldots-\geo{v}_k-\ldots-\geo{v}_{\frac{n}{4}} - \geo{v}_{\frac{n}{4}+1} - \ldots - \geo{v}_{\frac{n}{2}-k+1}-\ldots-\geo{v}_{\frac{n}{2}}-\geo{v}_0$ plus $n/2-1$ pendant edges: $\geo{v}_{0} - \geo{v}_{\frac{n}{2}+1}$, $\geo{v}_k - \geo{v}_{k + \frac{n}{2} + 1}$, $\geo{v}_{\frac{n}{2}-k+1} - \geo{v}_{n-k}$, $k = 1,\ldots, n/4-1$, as illustrated in Figure~\ref{figure:model}. We assume that $\geo{P}_n$ has the edge $\geo{v}_{0}-\geo{v}_{\frac{n}{2}+1}$ as axis of symmetry and for all $k=1,\ldots, n/4-1$, the pendant edge $\geo{v}_k - \geo{v}_{k + \frac{n}{2} + 1}$ bisects the angle $\angle \geo{v}_{k-1} \geo{v}_k \geo{v}_{k+1}$.

\begin{figure}
	\centering
	\begin{tikzpicture}[scale=5]		
		\draw[dashed] (0,0) node[below]{$\geo{v}_0(0,0)$} -- (0.2957,0.2043) node[below right]{$\geo{v}_7(x_7,y_7)$} -- (0.5000,0.5000) node[right]{$\geo{v}_3(x_3,y_3)$} -- (0.4114,0.9114) node[above right]{$\geo{v}_1(x_1,y_1)$} -- (0,1) node[above]{$\geo{v}_5(0,1)$} -- (-0.4114,0.9114) node[above left]{$\geo{v}_{4}(x_4,y_4)$} -- (-0.5000,0.5000) node[left]{$\geo{v}_2(x_2,y_2)$} -- (-0.2957,0.2043) node[below left]{$\geo{v}_6(x_6,y_6)$} -- cycle;
		\draw (0,0) -- (0.4114,0.9114) -- (-0.5000,0.5000) -- (0.5000,0.5000) -- (-0.4114,0.9114) -- cycle;
		\draw (0,0) -- (0,1);
		\draw (0.4114,0.9114) -- (-0.2957,0.2043);\draw (-0.4114,0.9114) -- (0.2957,0.2043);
		\draw (0.1029,0.2279) arc (65.71:90.00:0.25) node[midway,above]{${\color{red}\alpha_0}$};
		\draw (0.2716,0.7091) node{${\color{red}\alpha_1}$};
		\draw (0.2091,0.7716) node{${\color{red}\alpha_1}$};
		\draw (-0.3089,0.5411) node{${\color{red}\alpha_2}$};
	\end{tikzpicture}
	\caption{Definition of variables $\alpha_0, \alpha_1, \ldots, \alpha_{\frac{n}{4}}$ for $\geo{B}_n$: Case of $n=8$ vertices}
	\label{figure:model}
\end{figure}
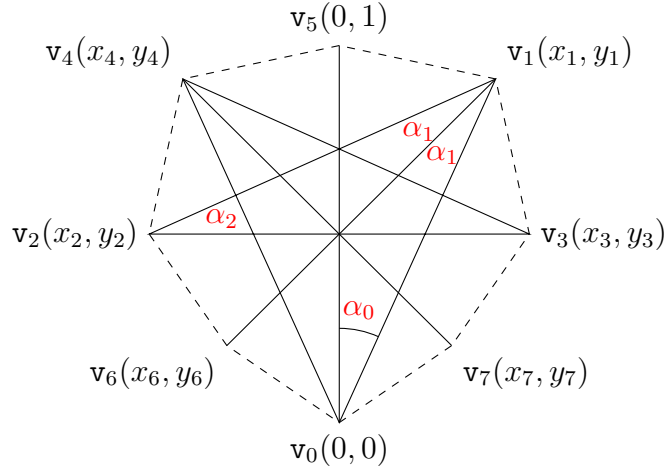

Let $\alpha_0 := \angle \geo{v}_{\frac{n}{2}+1} \geo{v}_{0} \geo{v}_1$, $2\alpha_k := \angle \geo{v}_{k-1} \geo{v}_{k} \geo{v}_{k+1}$ for all $k=1,\ldots, n/4-1$, and $\alpha_{\frac{n}{4}} := \angle \geo{v}_{\frac{n}{4}-1} \geo{v}_{\frac{n}{4}} \geo{v}_{\frac{n}{4}+1}$. Since $\geo{P}_n$ is symmetric, we have
\begin{equation}\label{eq:condition}
	\alpha_0 + 2\sum_{k=1}^{n/4-1}\alpha_k + \alpha_{n/4} = \frac{\pi}{2},
\end{equation}
and
\begin{subequations}\label{eq:LW}
	\begin{align}
		L(\geo{P}_n) &= 4\sin \frac{\alpha_0}{2} + 8 \sum_{k=1}^{n/4-1} \sin \frac{\alpha_k}{2} + 4\sin \frac{\alpha_{n/4}}{2},\label{eq:LW:L}\\
		W(\geo{P}_n) &= \min_{k=0,1,\ldots,n/4} \cos \frac{\alpha_k}{2}.
	\end{align}
\end{subequations}
By placing the vertex $\geo{v}_{0}$ at $(0,0)$ in the plane, and the vertex $\geo{v}_{\frac{n}{2}+1}$ at $(0,1)$, we have
\begin{subequations}\label{eq:xy}
	\begin{align}
		x_1 &= \sin \alpha_0 &&= -x_{\frac{n}{2}},\\
		y_1 &= \cos \alpha_0 &&= y_{\frac{n}{2}},\\
		x_k &= x_{k-1} - (-1)^k \sin \left(\alpha_0 + 2\sum_{j=1}^{k-1}\alpha_j\right) && = - x_{\frac{n}{2} - k + 1} &\forall k=2,3,\ldots,n/4\\
		y_k &= y_{k-1} - (-1)^k \cos \left(\alpha_0 + 2\sum_{j=1}^{k-1}\alpha_j\right) && = y_{\frac{n}{2} - k + 1} &\forall k=2,3,\ldots,n/4,\\
		x_{k+\frac{n}{2}+1} &= x_{k} + (-1)^k \sin \left(\alpha_0 + 2\sum_{j=1}^{k-1}\alpha_j + \alpha_k \right) && = - x_{n-k} &\forall k=1,2,\ldots,n/4-1,\\
		y_{k+\frac{n}{2}+1} &= y_{k} + (-1)^k \cos \left(\alpha_0 + 2\sum_{j=1}^{k-1}\alpha_j + \alpha_k \right) && = y_{n-k} &\forall k=1,2,\ldots,n/4-1.
	\end{align}
\end{subequations}
We also have
\begin{equation}\label{eq:x2}
	x_{\frac{n}{4}} = -1/2 = -x_{\frac{n}{4}+1}.
\end{equation}
since the edge $\geo{v}_{\frac{n}{4}} - \geo{v}_{\frac{n}{4}+1}$ is horizontal and $\|\geo{v}_{\frac{n}{4}} - \geo{v}_{\frac{n}{4}+1}\| = 1$.

For all $k=0,1,\ldots,n/4$, suppose $\alpha_k = \frac{\pi}{n} + (-1)^k \beta$ with $\beta = \beta(n)$ satisfying $|\beta| < \frac{\pi}{n}$. Then \eqref{eq:condition} is verified and \eqref{eq:LW} becomes
\begin{subequations}\label{eq:LWd}
	\begin{align}
		L(\geo{P}_n) &= n\sin \left(\frac{\pi}{2n}+\frac{\beta}{2}\right) + n\sin \left(\frac{\pi}{2n}-\frac{\beta}{2}\right) = 2n \sin \frac{\pi}{2n}\cos \frac{\beta}{2},\\
		W(\geo{P}_n) &= \cos \left(\frac{\pi}{2n} + \frac{|\beta|}{2}\right).
	\end{align}
\end{subequations}
Coordinates $(x_i,y_i)$ in~\eqref{eq:xy} are given by
\begin{subequations}\label{eq:xyd}
	\begin{align}
		x_{k} &= \sum_{j=1}^k (-1)^{j-1} \sin \left((2j-1)\frac{\pi}{n}+(-1)^{j-1}\beta\right) \nonumber\\
		&= \frac{\sin \frac{2k\pi}{n} \sin \left(\beta - (-1)^k\frac{\pi}{n}\right)}{\sin \frac{2\pi}{n}} &&= - x_{\frac{n}{2}-k+1} &\forall k=1,2,\ldots,n/4, \label{eq:xyd:3n4}\\
		y_{k} &= \sum_{j=1}^k (-1)^{j-1} \cos \left((2j-1)\frac{\pi}{n}+(-1)^{j-1}\beta\right) \nonumber\\
		&= \frac{\sin\left(\frac{\pi}{n}-\beta\right)+\cos \frac{2k\pi}{n}\sin \left(\beta - (-1)^k\frac{\pi}{n}\right)}{\sin \frac{2\pi}{n}} &&= y_{\frac{n}{2}-k+1} &\forall k=1,2,\ldots,n/4,\\
		x_{k+\frac{n}{2}+1} &= x_{k} + (-1)^k \sin \frac{2k\pi}{n} &&= - x_{n-k} &\forall k=1,2,\ldots,n/4-1,\\
		y_{k+\frac{n}{2}+1} &= y_{k} + (-1)^k \cos \frac{2k\pi}{n} &&= y_{n-k} &\forall k=1,2,\ldots,n/4-1.
	\end{align}
\end{subequations}

Finally, $\beta$ is chosen so that \eqref{eq:x2} is satisfied. It follows, from~\eqref{eq:xyd:3n4},
\[
\frac{\sin \left(\beta -\frac{\pi}{n}\right)}{\sin \frac{2\pi}{n}} = -\frac{1}{2} \Rightarrow \beta = \beta_0(n) = \frac{\pi}{n}-\arcsin \left(\frac{1}{2}\sin \frac{2\pi}{n}\right) = \frac{\pi^3}{2n^3} + \frac{\pi^5}{8n^5} + O\left(\frac{1}{n^7}\right).
\]
Let $\geo{B}_n$ denote the $n$-gon obtained by setting $\beta = \beta_0(n)$. From~\eqref{eq:LWd}, we have
\[
\begin{aligned}
	L(\geo{B}_n) &= 2n \sin \frac{\pi}{2n} \cos \left(\frac{\pi}{2n}-\frac{1}{2}\arcsin\left(\frac{1}{2}\sin \frac{2\pi}{n}\right)\right),\\
	W(\geo{B}_n) &= \cos \left(\frac{\pi}{n}-\frac{1}{2}\arcsin \left(\frac{1}{2}\sin \frac{2\pi}{n} \right)\right),
\end{aligned}
\]
and
\[
\begin{aligned}
	\ub{L}_n - L(\geo{B}_n) &= \frac{\pi^7}{32n^6} + \frac{11\pi^9}{768n^8} + O\left(\frac{1}{n^{10}}\right),\\
	\ub{W}_n - W(\geo{B}_n) &= \frac{\pi^4}{8n^4} + \frac{11\pi^6}{192n^6} + O\left(\frac{1}{n^{8}}\right).
\end{aligned}
\]
By construction, $\geo{B}_n$ is small and convex for all $n=2^s$ and $s \ge 3$. We illustrate $\geo{B}_n$ for some $n$ in Figure~\ref{figure:Bn}. This completes the proof of Theorem~\ref{thm:Bn}.\qed

\begin{figure}
	\centering
	\subfloat[$(\geo{B}_8,3.1211,0.9776)$]{
		\begin{tikzpicture}[scale=4]
			\draw[dashed] (0,0) -- (0.2957,0.2043) -- (0.5000,0.5000) -- (0.4114,0.9114) -- (0,1) -- (-0.4114,0.9114) -- (-0.5000,0.5000) -- (-0.2957,0.2043) -- cycle;
			\draw[blue,thick] (0,0) -- (0.4114,0.9114) -- (-0.5000,0.5000) -- (0.5000,0.5000) -- (-0.4114,0.9114) -- cycle;
			\draw[red,thick] (0,0) -- (0,1);
			\draw (0.4114,0.9114) -- (-0.2957,0.2043);\draw (-0.4114,0.9114) -- (0.2957,0.2043);
		\end{tikzpicture}
	}
	\subfloat[$(\geo{B}_{16},3.1365,0.9950)$]{
		\begin{tikzpicture}[scale=4]
			\draw[dashed] (0,0) -- (0.1838,0.0562) -- (0.3536,0.1464) -- (0.4438,0.3162) -- (0.5000,0.5000) -- (0.4800,0.6988) -- (0.3536,0.8536) -- (0.1988,0.9800) -- (0,1) -- (-0.1988,0.9800) -- (-0.3536,0.8536) -- (-0.4800,0.6988) -- (-0.5000,0.5000) -- (-0.4438,0.3162) -- (-0.3536,0.1464) -- (-0.1838,0.0562) -- cycle;
			\draw[blue,thick] (0,0) -- (0.1988,0.9800) -- (-0.3536,0.1464) -- (0.4800,0.6988) -- (-0.5000,0.5000) -- (0.5000,0.5000) -- (-0.4800,0.6988) -- (0.3536,0.1464) -- (-0.1988,0.9800) -- cycle;
			\draw[red,thick] (0,0) -- (0,1);
			\draw (0.1988,0.9800) -- (-0.1838,0.0562);\draw (-0.1988,0.9800) -- (0.1838,0.0562);
			\draw (-0.3536,0.1464) -- (0.3536,0.8536);\draw (0.3536,0.1464) -- (-0.3536,0.8536);
			\draw (0.4800,0.6988) -- (-0.4438,0.3162);\draw (-0.4800,0.6988) -- (0.4438,0.3162);
		\end{tikzpicture}
	}
	\subfloat[$(\geo{B}_{32},3.1403,0.9988)$]{
		\begin{tikzpicture}[scale=4]
			\draw[dashed] (0,0) -- (0.0966,0.0144) -- (0.1913,0.0381) -- (0.2751,0.0883) -- (0.3536,0.1464) -- (0.4117,0.2249) -- (0.4619,0.3087) -- (0.4856,0.4034) -- (0.5000,0.5000) -- (0.4951,0.5985) -- (0.4619,0.6913) -- (0.4198,0.7805) -- (0.3536,0.8536) -- (0.2805,0.9198) -- (0.1913,0.9619) -- (0.0985,0.9951) -- (0,1) -- (-0.0985,0.9951) -- (-0.1913,0.9619) -- (-0.2805,0.9198) -- (-0.3536,0.8536) -- (-0.4198,0.7805) -- (-0.4619,0.6913) -- (-0.4951,0.5985) -- (-0.5000,0.5000) -- (-0.4856,0.4034) -- (-0.4619,0.3087) -- (-0.4117,0.2249) -- (-0.3536,0.1464) -- (-0.2751,0.0883) -- (-0.1913,0.0381) -- (-0.0966,0.0144) -- cycle;
			\draw[blue,thick] (0,0) -- (0.0985,0.9951) -- (-0.1913,0.0381) -- (0.2805,0.9198) -- (-0.3536,0.1464) -- (0.4198,0.7805) -- (-0.4619,0.3087) -- (0.4951,0.5985) -- (-0.5000,0.5000) -- (0.5000,0.5000) -- (-0.4951,0.5985) -- (0.4619,0.3087) -- (-0.4198,0.7805) -- (0.3536,0.1464) -- (-0.2805,0.9198) -- (0.1913,0.0381) -- (-0.0985,0.9951) -- cycle;
			\draw[red,thick] (0,0) -- (0,1);
			\draw (0.0985,0.9951) -- (-0.0966,0.0144);\draw (-0.0985,0.9951) -- (0.0966,0.0144);
			\draw (-0.1913,0.0381) -- (0.1913,0.9619);\draw (0.1913,0.0381) -- (-0.1913,0.9619);
			\draw (0.2805,0.9198) -- (-0.2751,0.0883);\draw (-0.2805,0.9198) -- (0.2751,0.0883);
			\draw (-0.3536,0.1464) -- (0.3536,0.8536);\draw (0.3536,0.1464) -- (-0.3536,0.8536);
			\draw (0.4198,0.7805) -- (-0.4117,0.2249);\draw (-0.4198,0.7805) -- (0.4117,0.2249);
			\draw (-0.4619,0.3087) -- (0.4619,0.6913);\draw (0.4619,0.3087) -- (-0.4619,0.6913);
			\draw (0.4951,0.5985) -- (-0.4856,0.4034);\draw (-0.4951,0.5985) -- (0.4856,0.4034);
		\end{tikzpicture}
	}
	\caption{Polygons $(\geo{B}_n,L(\geo{B}_n),W(\geo{B}_n))$ defined in Theorem~\ref{thm:Bn}}
	\label{figure:Bn}
\end{figure}
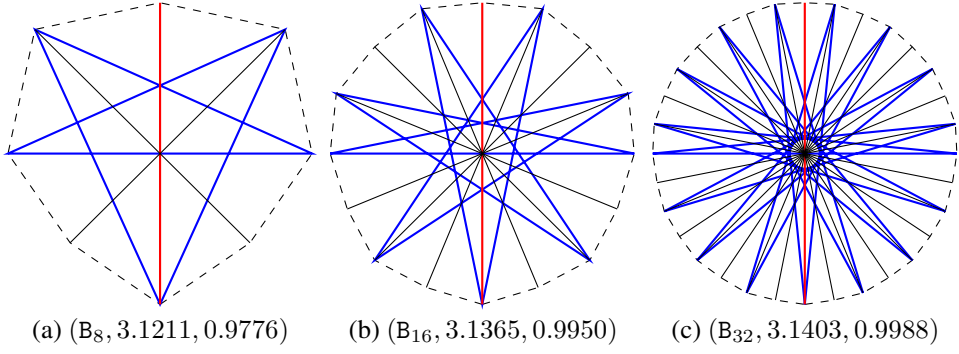

We implemented all polygons presented in this work as a MATLAB package: OPTIGON, freely available on GitHub~\cite{optigon}. In OPTIGON, MATLAB functions that give the coordinates of the vertices are provided. An algorithm developed in~\cite{bingane2022d} to estimate the maximal area of a small $n$-gon~\cite{reinhardt1922} when $n \ge 6$ is even can be also found.

Table~\ref{table:perimeter} shows the perimeters of $\geo{B}_n$, along with the upper bounds $\ub{L}_n$, the perimeters of $\geo{R}_n$, $\geo{R}_{n-1}^+$, $\geo{T}_n$, and $\geo{M}_n$ for $n=2^s$ and $3 \le s \le 7$. As suggested by Theorem~\ref{thm:Bn}, when $n$ is a power of $2$, $\geo{B}_n$ provides a tighter lower bound on the maximal perimeter~$L_n^*$ compared to the best prior convex small $n$-gon~$\geo{M}_n$. For instance, we can note that $L_{128}^* - L(\geo{B}_{128}) < \ub{L}_{128} - L(\geo{B}_{128}) < 2.15 \times 10^{-11}$. By analysing the fraction $\frac{L(\geo{B}_n) - L(\geo{M}_n)}{\ub{L}_n - L(\geo{M}_n)}$ of the length of the interval $[L(\geo{M}_n), \ub{L}_n]$ where $L(\geo{B}_n)$ lies, it is not surprising that $L(\geo{B}_n)$ approaches $\ub{L}_n$ much faster than $L(\geo{M}_n)$ does as $n$ increases. After all, $L(\geo{B}_n) - L(\geo{M}_n) \sim \frac{\pi^5}{16n^5}$ for large $n$.

\begin{table}
	\footnotesize
	\centering
	\caption{Perimeters of $\geo{B}_n$}
	\label{table:perimeter}
		\resizebox{\linewidth}{!}{
		\begin{tabular}{@{}rllllllr@{}}
			\toprule
			$n$ & $L(\geo{R}_n)$ & $L(\geo{R}_{n-1}^+)$ & $L(\geo{T}_n)$ & $L(\geo{M}_n)$ & $L(\geo{B}_n)$ & $\ub{L}_n$ & $\frac{L(\geo{B}_n)- L(\geo{M}_n)}{\ub{L}_n - L(\geo{M}_n)}$ \\
			\midrule
			8	&	3.0614674589	&	3.1181091119	&	3.1190543124	&	3.1209757852	&	3.1210621230	&	3.1214451523	&	0.1839	\\
			16	&	3.1214451523	&	3.1361407965	&	3.1364381783	&	3.1365320240	&	3.1365427675	&	3.1365484905	&	0.6524	\\
			32	&	3.1365484905	&	3.1402809876	&	3.1403234211	&	3.1403306141	&	3.1403310687	&	3.1403311570	&	0.8374	\\
			64	&	3.1403311570	&	3.1412710339	&	3.1412767980	&	3.1412772335	&	3.1412772496	&	3.1412772509	&	0.9211	\\
			128	&	3.1412772509	&	3.1415130275	&	3.1415137720	&	3.1415138006	&	3.141513801123	&	3.141513801144	&	0.9606	\\
			\bottomrule
		\end{tabular}
			}
\end{table}

Table~\ref{table:width} displays the widths of $\geo{B}_n$, along with the upper bounds $\ub{W}_n$, the widths of $\geo{R}_n$ and $\geo{R}_{n-1}^+$. Again, when $n=2^s$, $\geo{B}_n$ provides a tighter lower bound for the maximal width~$W_n^*$ compared to the best prior convex small $n$-gon~$\geo{R}_{n-1}^+$. We also remark that $W(\geo{B}_n)$ approaches $\ub{W}_n$ much faster than $W(\geo{R}_{n-1}^+)$ does as $n$ increases. It is interesting to note that $W(\geo{B}_8) = W_8^*$, i.e., $\geo{B}_8$ is an optimal solution for the maximal width problem when $n=8$.

\begin{table}
	\footnotesize
	\centering
	\caption{Widths of $\geo{B}_n$}
	\label{table:width}
		\begin{tabular}{@{}rllllr@{}}
			\toprule
			$n$ & $	W(\geo{R}_n)$ & $W(\geo{R}_{n-1}^+)$ & $W(\geo{B}_n)$ & $\ub{W}_n$ & $\frac{W(\geo{B}_n) - W(\geo{R}_{n-1}^+)}{\ub{W}_n - W(\geo{R}_{n-1}^+)}$ \\
			\midrule
			8	&	0.9238795325	&	0.9749279122	&	0.9776087734	&	0.9807852804	&	0.4577	\\
			16	&	0.9807852804	&	0.9945218954	&	0.9949956687	&	0.9951847267	&	0.7148	\\
			32	&	0.9951847267	&	0.9987165072	&	0.9987837929	&	0.9987954562	&	0.8523	\\
			64	&	0.9987954562	&	0.9996891820	&	0.9996980921	&	0.9996988187	&	0.9246	\\
			128	&	0.9996988187	&	0.9999235114	&	0.9999246565	&	0.9999247018	&	0.9619	\\
			\bottomrule
		\end{tabular}
\end{table}

Proposition~\ref{thm:Bn:xy} and Proposition~\ref{thm:Bn:area} highlight some interesting properties of $\geo{B}_n$.

\begin{proposition}\label{thm:Bn:xy}
	Let $n=2^s$ with integer $s\ge 3$.
	\begin{enumerate}
		\item The coordinates of the vertex $\geo{v}_{\frac{n}{4}}$ in $\geo{B}_n$ are $(-1/2,1/2)$.
		\item For all $k=1,\ldots,n/4-1$, the pendant edge $\geo{v}_{k} - \geo{v}_{k+\frac{n}{2}+1}$ of $\geo{B}_n$ passes through the point $\geo{u}=(0,1/2)$.
	\end{enumerate}
\end{proposition}
\begin{proof}
	Let $n=2^s$ with integer $s\ge 3$ and $\beta = \frac{\pi}{n}-\arcsin \left(\frac{1}{2}\sin \frac{2\pi}{n}\right)$.
	\begin{enumerate}
		\item We have, from~\eqref{eq:xyd:3n4},
		\[
		\begin{aligned}
			x_{\frac{n}{4}} &= \frac{\sin\left(\beta - \frac{\pi}{n}\right)}{\sin \frac{2\pi}{n}} = -\frac{1}{2},\\
			y_{\frac{n}{4}} &= \frac{\sin\left(\frac{\pi}{n} - \beta\right)}{\sin \frac{2\pi}{n}} = \frac{1}{2}.
		\end{aligned}
		\]
		\item For all $k=1,\ldots,n/4-1$, coordinates $(x_i,y_i)$ in~\eqref{eq:xyd} are
		\[
		\begin{aligned}
			x_{k} &= \frac{\sin \frac{2k\pi}{n} \sin \left(\beta - (-1)^k\frac{\pi}{n}\right)}{\sin \frac{2\pi}{n}}, &
			x_{k+\frac{n}{2}+1} &= x_{k} + (-1)^k \sin \frac{2k\pi}{n},\\
			y_{k} &= \frac{1}{2} + \frac{\cos \frac{2k\pi}{n}\sin \left(\beta - (-1)^k\frac{\pi}{n}\right)}{\sin \frac{2\pi}{n}}, & y_{k+\frac{n}{2}+1} &= y_{k} + (-1)^k \cos \frac{2k\pi}{n}.
		\end{aligned}
		\]
		It follows that, for all $k=1,\ldots,n/4-1$,
		\[
		\frac{x_{k+\frac{n}{2}+1} - x_{k}}{y_{k+\frac{n}{2}+1} - y_{k}} = \tan \frac{2k\pi}{n} = \frac{x_{k}}{y_{k} - \frac{1}{2}},
		\]
		i.e., the pendant edge $\geo{v}_{k} - \geo{v}_{k+\frac{n}{2}+1}$ passes through the point $\geo{u}=(0,1/2)$.
	\end{enumerate}
\end{proof}

\begin{proposition}\label{thm:Bn:area}
	Let $n=2^s$ with integer $s\ge 3$. The area of $\geo{B}_n$ is $\frac{n}{8} \sin \frac{2\pi}{n}$, which is the area of the regular small $n$-gon $\geo{R}_n$.
\end{proposition}
\begin{proof}
	Let $n=2^s$ with integer $s\ge 3$ and $\beta = \frac{\pi}{n}-\arcsin \left(\frac{1}{2}\sin \frac{2\pi}{n}\right)$. Let $A_0$ be the area of the quadrilateral $\geo{u} \geo{v}_{1} \geo{v}_{\frac{n}{2}+1} \geo{v}_{\frac{n}{2}}$, $A_k$~be the area of the quadrilateral $\geo{u} \geo{v}_{k+1} \geo{v}_{k+\frac{n}{2}+1} \geo{v}_{k-1}$ for all $k=1,\ldots,n/4-1$, and $A_{\frac{n}{4}}$~be the area of the triangle $\geo{u} \geo{v}_{\frac{n}{4}+1} \geo{v}_{\frac{n}{4}-1}$, where $\geo{u}=(0,1/2)$. The area of~$\geo{B}_n$ is given by
	\[
	A(\geo{B}_n) = A_0 + 2\sum_{k=1}^{n/4-1}A_k + 2A_{\frac{n}{4}}.
	\]
	
	We have
	\[
	\begin{aligned}
		A_0 &= \frac{1}{2}\|\geo{v}_{\frac{n}{2}+1}-\geo{u}\| \|\geo{v}_{\frac{n}{2}} - \geo{v}_{1}\| = \frac{1}{2}\sin \left(\frac{\pi}{n}+\beta\right),\\
		A_k &= \frac{1}{2}\|\geo{v}_{k+\frac{n}{2}+1}-\geo{u}\| \|\geo{v}_{k-1} - \geo{v}_{k+1}\|\\
		&=
		\begin{cases}
			\frac{1}{2}\sin \left(\frac{\pi}{n}+\beta\right) &\text{if $k$ is even},\\
			\frac{1}{2}\sin\frac{2\pi}{n} - \frac{1}{2}\sin \left(\frac{\pi}{n}+\beta\right) &\text{if $k$ is odd},
		\end{cases}
	\end{aligned}
	\]
	for all $k=1,\ldots,n/4-1$, and
	\[
	A_{\frac{n}{4}} = \frac{1}{2}(x_{\frac{n}{4}+1}(y_{\frac{n}{4}-1}-1/2) - (y_{\frac{n}{4}+1}-1/2) x_{\frac{n}{4}-1}) = \frac{1}{4} \sin \left(\frac{\pi}{n}+\beta\right).
	\]
	Thus,
	\[
	A(\geo{B}_n) = \frac{n}{8}\sin \left(\frac{\pi}{n}+\beta\right) + \frac{n}{8}\left(\sin\frac{2\pi}{n} - \sin \left(\frac{\pi}{n}+\beta\right)\right) = \frac{n}{8}\sin\frac{2\pi}{n}.
	\]
\end{proof}

\section{Tight bounds on the maximal width of unit-perimeter polygons}\label{sec:wn}
Let $\hat{\geo{P}}$ denote the polygon obtained by contracting a small polygon $\geo{P}$ so that $L(\hat{\geo{P}}) = 1$. Thus, the width of the unit-perimeter polygon~$\hat{\geo{P}}$ is given by $W(\hat{\geo{P}}) = W(\geo{P})/L(\geo{P})$. For a given integer $n\ge 3$,
\[
W(\hat{\geo{R}}_n) =
\begin{cases}
	\frac{1}{2n} \cot \frac{\pi}{2n} &\text{if $n$ is odd,}\\
	\frac{1}{n} \cot \frac{\pi}{n} &\text{if $n$ is even.}\\
\end{cases}
\]
We remark that $W(\hat{\geo{R}}_n) < W(\hat{\geo{R}}_{n-1})$ for all even $n\ge 4$. The polygon $\hat{\geo{R}}_n$ does not have maximum width for any even $n\ge 4$. When $n$ is even, one can construct a unit-perimeter $n$-gon with the same width as $\hat{\geo{R}}_{n-1}$ by adding a vertex in the middle of a side of $\hat{\geo{R}}_{n-1}$.

When $n$ has an odd factor $m$, one can note that
\[
W(\hat{\geo{R}}_{m,n}) = \frac{1}{2n} \cot \frac{\pi}{2n}.
\]

\begin{theorem}[Audet, Hansen, and Messine~\cite{audet2009b}]\label{thm:width:perimeter}
	For all $n \ge 3$, let $w_n^*$ denote the maximal width among all unit-perimeter $n$-gons and let $\ub{w}_n := \frac{1}{2n} \cot \frac{\pi}{2n}$.
	\begin{itemize}
		\item When $n$ has an odd factor $m$, $w_n^* = \ub{w}_n$ is achieved by finitely many equilateral $n$-gons~\cite{mossinghoff2011,hare2013,hare2019}, including~$\hat{\geo{R}}_{m,n}$. The optimal $n$-gon~$\hat{\geo{R}}_{m,n}$ is unique if $m$ is prime and $n/m \le 2$.
		\item When $n=2^s$ with integer $s\ge 2$, $W(\hat{\geo{R}}_n) < \ub{w}_{n-1} \le w_n^* < \ub{w}_n$.
	\end{itemize}
\end{theorem}

When $n=2^s$, the maximal width $w_n^*$ of unit-perimeter $n$-gons is only known for $s=2$. Audet, Hansen, and Messine~\cite{audet2009b} showed that $w_4^* = \frac{1}{4}\sqrt{6\sqrt{3}-9} > \ub{w}_3 = \frac{1}{6}\sqrt{3}$. For $s\ge 3$, exact solutions appear to be presently out of reach. However, it is interesting to note that
\[
W(\hat{\geo{B}}_n) = \frac{1}{2n}\left(\cot \frac{\pi }{2n}  - \tan \left(\frac{\pi}{2n}-\frac{1}{2}\arcsin\left(\frac{1}{2}\sin \frac{2\pi}{n}\right)\right)\right)
\]
is a tighter lower bound compared to $\ub{w}_{n-1}$ on $w_n^*$ when $n=2^s$ and $s\ge 3$. Indeed, we can show that, for all $n=2^s$ and integer $s\ge 3$,
\[
\ub{w}_n - W(\hat{\geo{B}}_n) = \frac{1}{2n}\tan \left(\frac{\pi}{2n}-\frac{1}{2}\arcsin\left(\frac{1}{2}\sin \frac{2\pi}{n}\right)\right) = \frac{\pi^3}{8n^4} + O\left(\frac{1}{n^6}\right),
\]
while
\[
\begin{aligned}
	\ub{w}_n - W(\hat{\geo{R}}_n) &= \frac{\pi}{4n^2} + O\left(\frac{1}{n^4}\right),\\
	\ub{w}_n - \ub{w}_{n-1} &= \frac{\pi}{6n^3} + O\left(\frac{1}{n^4}\right)
\end{aligned}
\]
for all even $n\ge 4$.

Table~\ref{table:width:perimeter} lists the widths of $\hat{\geo{B}}_n$, along with the upper bounds $\ub{w}_n$, the lower bounds $\ub{w}_{n-1}$, and the widths of $\hat{\geo{R}}_n$ for $n=2^s$ and $3\le s \le 7$. As $n$ increases, it is not surprising that $W(\hat{\geo{B}}_n)$ approaches $\ub{W}_n$ much faster than $\ub{w}_{n-1}$ does.

\begin{table}
	\footnotesize
	\centering
	\caption{Widths of $\hat{\geo{B}}_n$}
	\label{table:width:perimeter}
		\begin{tabular}{@{}rllllr@{}}
			\toprule
			$n$ & $	W(\hat{\geo{R}}_n)$ & $\ub{w}_{n-1}$ & $W(\hat{\geo{B}}_n)$ & $\ub{w}_n$ & $\frac{W(\hat{\geo{B}}_n) - \ub{w}_{n-1}}{\ub{w}_n - \ub{w}_{n-1}}$ \\
			\midrule
			8	&	0.3017766953	&	0.3129490191	&	0.3132295145	&	0.3142087183	&	0.2227	\\
			16	&	0.3142087183	&	0.3171454818	&	0.3172268776	&	0.3172865746	&	0.5769	\\
			32	&	0.3172865746	&	0.3180374156	&	0.3180504765	&	0.3180541816	&	0.7790	\\
			64	&	0.3180541816	&	0.3182439224	&	0.3182457366	&	0.3182459678	&	0.8870	\\
			128	&	0.3182459678	&	0.3182936544	&	0.3182938926	&	0.3182939071	&	0.9428	\\
			\bottomrule
		\end{tabular}
\end{table}

\section{Solving the maximal perimeter problem}\label{sec:nlo}

For any $n=2^s$ with integer $s\ge 3$, we can construct a convex small $n$-gon $\geo{B}_n^*$ with a longer perimeter than $\geo{B}_n$ by adjusting the angles $\alpha_0, \alpha_1, \ldots, \alpha_{\frac{n}{4}}$ from the parametrization of Section~\ref{sec:Bn} to maximize the perimeter $L(\geo{P}_n)$ in~\eqref{eq:LW:L}~\cite{mossinghoff2006b}. Hence, $L(\geo{B}_n^*)$ is the optimal value of the problem:
\begin{subequations}\label{eq:ngon:Bn}
	\begin{align}
		L(\geo{B}_n^*) = \max_{\rv{\alpha}} \quad & L(\geo{P}_n) = 4\sin \frac{\alpha_0}{2} + \sum_{k=1}^{n/4-1} 8\sin \frac{\alpha_k}{2} + 4\sin \frac{\alpha_{n/4}}{2}\\
		\subj \quad & \alpha_0 + \sum_{k=1}^{n/4-1} 2\alpha_k + \alpha_{n/4}  = \pi/2,\\
		& \sin \alpha_0 - \sum_{k=2}^{n/4} (-1)^k \sin \left(\alpha_0 + \sum_{i=1}^{k-1} 2\alpha_i\right) = -1/2,\\
		& 0 \le \alpha_k \le \pi/6 \quad \forall k = 0,1, \ldots, n/4-1,\\
		& 0 \le \alpha_{n/4} \le \pi/3,\\
		& L(\geo{P}_n) \ge L(\geo{B}_n).
	\end{align}
\end{subequations}
This formulation was used in~\cite{audet2007a} for $n=8$ to find the convex small $8$-gon of maximal perimeter.

For each $n = 2^s$ with integer $s \ge 2$, one can also construct a convex small $n$-gon $\geo{Q}_n^*$ with the same diameter graph as $\geo{R}_{n-1}^+$ but larger perimeter. Using a similar parametrization as in Section~\ref{sec:Bn}, we can show that
\begin{subequations}\label{eq:ngon:Qn}
	\begin{align}
		L(\geo{Q}_n^*) = \max_{\rv{\alpha}} \quad & L(\geo{P}_n) = \sum_{k=0}^{n/2-1} 4\sin \frac{\alpha_k}{2}\\
		\subj \quad & \sum_{k=0}^{n/2-1} \alpha_k  = \pi/2,\\
		& \sum_{k=0}^{n/2-2} (-1)^k \sin \left(\sum_{i=0}^{k} \alpha_i\right) = 1/2,\\
		& 0 \le \alpha_0 \le \pi/6,\\
		& 0 \le \alpha_k \le \pi/3 \quad \forall k = 1,2, \ldots, n/2-1,\\
		& L(\geo{P}_n) \ge L(\geo{R}_{n-1}^+).
	\end{align}
\end{subequations}
The variables $\alpha_0,\alpha_1,\ldots,\alpha_{\frac{n}{2}-1}$ are defined in Figure~\ref{figure:model:Qn}. Clearly, $\geo{Q}_4^* \equiv \geo{R}_3^+$ and $L(\geo{Q}_4^*) = L_4^*$.

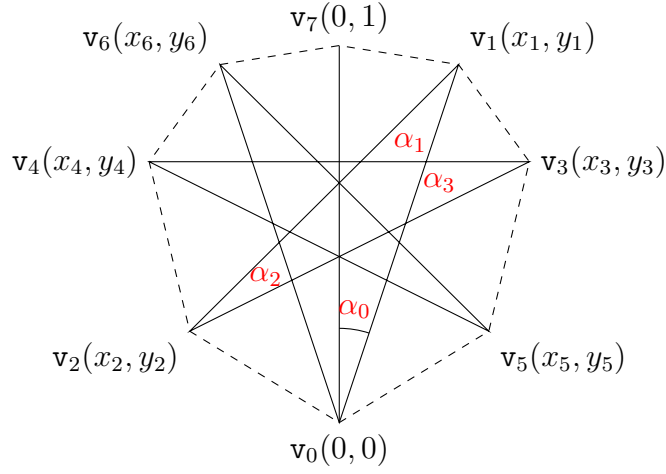
\begin{figure}
	\centering
	\begin{tikzpicture}[scale=5]
		\draw[dashed] (0,0) node[below]{$\geo{v}_0(0,0)$} -- (0.3933,0.2424) node[below right]{$\geo{v}_5(x_5,y_5)$} -- (0.5000,0.6919) node[right]{$\geo{v}_3(x_3,y_3)$} -- (0.3138,0.9495) node[above right]{$\geo{v}_1(x_1,y_1)$} -- (0,1) node[above]{$\geo{v}_7(0,1)$} -- (-0.3138,0.9495) node[above left]{$\geo{v}_{6}(x_6,y_6)$} -- (-0.5000,0.6919) node[left]{$\geo{v}_4(x_4,y_4)$} -- (-0.3933,0.2424) node[below left]{$\geo{v}_2(x_2,y_2)$} -- cycle;
		\draw ((0,0) -- (0.3138,0.9495) -- (-0.3933,0.2424) -- (0.5000,0.6919) -- (-0.5000,0.6919) -- (0.3933,0.2424) -- (-0.3138,0.9495) -- cycle;
		\draw (0,0) -- (0,1);
		\draw (0.0785,0.2374) arc (71.71:90.00:0.25) node[midway,above]{${\color{red}\alpha_0}$};
		\draw (0.1862,0.7424) node{${\color{red}\alpha_1}$};
		\draw (-0.1933,0.3870) node{${\color{red}\alpha_2}$};
		\draw (0.2633,0.6357) node{${\color{red}\alpha_3}$};
	\end{tikzpicture}
	\caption{Variables $\alpha_0,\alpha_1,\ldots,\alpha_{\frac{n}{2}-1}$ for $L(\geo{Q}_n^*)$: Case of $n=8$ vertices}
	\label{figure:model:Qn}
\end{figure}

We solved both Problems~\eqref{eq:ngon:Bn} and~\eqref{eq:ngon:Qn} on the NEOS Server~6.0 using AMPL with the solver Couenne~0.5.8, which is a branch-and-bound algorithm that aims at finding global optima of nonconvex mixed-integer nonlinear optimization problems~\cite{belotti2009}. We have made AMPL codes available in OPTIGON~\cite{optigon}.

Table~\ref{table:perimeter:optimal} gives the optimal values $L(\geo{B}_n^*)$ and $L(\geo{Q}_n^*)$ for $n=2^s$ and $3 \le s \le 7$, along with the perimeters of $\geo{B}_n$ and the upper bounds~$\ub{L}_n$. Couenne took less than 1 second to compute each $L(\geo{B}_n^*)$ or $L(\geo{Q}_n^*)$ except for $L(\geo{Q}_{16}^*)$, which was computed in 36 minutes. The results in Table~\ref{table:perimeter:optimal} support the following key points:
\begin{enumerate}
	\item The optimal perimeter $L(\geo{B}_n^*)$ for each $n \le 64$ computed agrees with the best value found in the literature.
	\item For all $n = 2^s$ and $s \ge 3$, $L(\geo{Q}_n^*) < L(\geo{B}_n) < L(\geo{B}_n^*)$, i.e., $\geo{Q}_n^*$ is a suboptimal solution.
	\item As $n$ increases, the fraction $\frac{L(\geo{B}_n^*) - L(\geo{B}_n)}{\ub{L}_n-L(\geo{B}_n)}$ appears to approach a scalar $b^* \in (0,1)$, i.e., $\ub{L}_n - L(\geo{B}_n^*) = O(1/n^6)$.
	\item For $n=8$, $L(\geo{B}_8^*) = L_8^*$.
\end{enumerate}

\begin{table}
	\footnotesize
	\centering
	\caption{Perimeters of $\geo{B}_n^*$ and $\geo{Q}_n^*$}
	\label{table:perimeter:optimal}
		\begin{tabular}{@{}rllllr@{}}
			\toprule
			$n$ & $L(\geo{Q}_n^*)$ & $L(\geo{B}_n)$ & $L(\geo{B}_n^*)$ & $\ub{L}_n$ & $\frac{L(\geo{B}_n^*) - L(\geo{B}_n)}{\ub{L}_n-L(\geo{B}_n)}$ \\
			\midrule
			8	&	3.1195976652~\cite{mossinghoff2006b}	&	3.1210621230	&	3.1211471341~\cite{mossinghoff2006b,audet2007a}	&	3.1214451523	&	0.2219	\\
			16	&	3.1364309268	&	3.1365427675	&	3.1365439563~\cite{mossinghoff2006b}	&	3.1365484905	&	0.2077	\\
			32	&	3.1403237758	&	3.1403310687	&	3.1403310858~\cite{mossinghoff2006b}	&	3.1403311570	&	0.1945	\\
			64	&	3.1412767891	&	3.1412772496	&	3.1412772498~\cite{mossinghoff2008}	&	3.1412772509	&	0.1907	\\
			128	&	3.1415137723	&	3.141513801123	&	3.141513801127	&	3.141513801144	&	0.1899	\\
			\bottomrule
		\end{tabular}
\end{table}

\begin{table}
	\footnotesize
	\centering
	\caption{Angles $\alpha_0^*, \alpha_1^*, \ldots, \alpha_{\frac{n}{4}}^*$ of $\geo{B}_n^*$}
	\label{table:angles:Bn}
			\resizebox{\linewidth}{!}{
		\begin{tabular}{@{}rl|rllllllll@{}}
			\toprule
			$n$ & $\pi/n$ & $i$ & $\alpha_{8i}^*$ & $\alpha_{8i+1}^*$ & $\alpha_{8i+2}^*$ & $\alpha_{8i+3}^*$ & $\alpha_{8i+4}^*$ & $\alpha_{8i+5}^*$ & $\alpha_{8i+6}^*$ & $\alpha_{8i+7}^*$ \\
			\midrule
			8	&	0.392699	& 0 &	0.435281	&	0.368535	&	0.398447	&&&&&\\
			16	&	0.196350	& 0 &	0.201226	&	0.191978	&	0.199873	&	0.194672	&	0.196525	&&&\\
			32	&	0.0981748	& 0 &	0.0987786	&	0.0975863	&	0.0987333	&	0.0976772	&	0.0986041	&	0.0978448	&	0.0984101	&	0.0980628	\\
			&& 1 &	0.0981803	&&&&&&&\\
			64	&	0.0490874	& 0 &	0.0491627	&	0.0490125	&	0.0491613	&	0.0490154	&	0.049157	&	0.0490211	&	0.0491501	&	0.0490293	\\
			&& 1 &	0.0491407	&	0.0490398	&	0.0491293	&	0.049052	&	0.0491164	&	0.0490657	&	0.0491022	&	0.0490802	\\
			&& 2 &	0.0490876	&&&&&&&\\
			128	&	0.0245437	& 0 &	0.0245531	&	0.0245343	&	0.0245531	&	0.0245344	&	0.0245529	&	0.0245346	&	0.0245527	&	0.0245348	\\
			&& 1 &	0.0245524	&	0.0245352	&	0.024552	&	0.0245356	&	0.0245515	&	0.0245361	&	0.024551	&	0.0245367	\\
			&& 2 &	0.0245504	&	0.0245374	&	0.0245497	&	0.0245381	&	0.0245489	&	0.0245389	&	0.0245481	&	0.0245397	\\
			&& 3 &	0.0245473	&	0.0245405	&	0.0245464	&	0.0245414	&	0.0245455	&	0.0245423	&	0.0245446	&	0.0245432	\\
			&& 4 &	0.0245437	&&&&&&&\\
			\bottomrule
		\end{tabular}
				}
\end{table}

\begin{table}
	\footnotesize
	\centering
	\caption{Angles $\alpha_0^*, \alpha_1^*, \ldots, \alpha_{\frac{n}{2}-1}^*$ of $\geo{Q}_n^*$}
	\label{table:angles:Qn}
			\resizebox{\linewidth}{!}{
		\begin{tabular}{@{}rl|rllllllll@{}}
			\toprule
			$n$ & $\pi/n$ & $i$ & $\alpha_{8i}^*$ & $\alpha_{8i+1}^*$ & $\alpha_{8i+2}^*$ & $\alpha_{8i+3}^*$ & $\alpha_{8i+4}^*$ & $\alpha_{8i+5}^*$ & $\alpha_{8i+6}^*$ & $\alpha_{8i+7}^*$ \\
			\midrule
			8	&	0.392699	&	0	&	0.301375	&	0.480058	&	0.355776	&	0.433588	&&&&\\
			16	&	0.196350	&	0	&	0.172189	&	0.219956	&	0.175546	&	0.216429	&	0.182713	&	0.210185	&	0.192054	&	0.201725	\\
			32	&	0.0981748	&	0	&	0.0920622	&	0.104242	&	0.0922572	&	0.103986	&	0.0927078	&	0.103531	&	0.0933908	&	0.102886	\\
			&&	1	&	0.0942718	&	0.10207	&	0.0953079	&	0.101105	&	0.0964509	&	0.100022	&	0.0976502	&	0.988561	\\
			64	&	0.0490874	& 0 &	0.0475548	&	0.0506167	&	0.0475665	&	0.0505995	&	0.0475937	&	0.0505686	&	0.0476362	&	0.0505242	\\
			&& 1 &	0.0476935	&	0.0504667	&	0.0477649	&	0.0503966	&	0.0478497	&	0.0503143	&	0.0479468	&	0.0502207	\\
			&& 2 &	0.0480553	&	0.0501163	&	0.0481739	&	0.0500023	&	0.0483013	&	0.0498794	&	0.0484363	&	0.0497487	\\
			&& 3 &	0.0485773	&	0.0496115	&	0.048723	&	0.0494689	&	0.0488718	&	0.0493222	&	0.0490222	&	0.0491729	\\
			128	&	0.0245437	& 0 &	0.0241605	&	0.0249269	&	0.024161	&	0.0249258	&	0.0241627	&	0.0249237	&	0.0241653	&	0.0249209	\\
			&& 1 &	0.0241688	&	0.0249171	&	0.0241733	&	0.0249124	&	0.0241787	&	0.0249069	&	0.024185	&	0.0249005	\\
			&& 2 &	0.0241922	&	0.0248933	&	0.0242002	&	0.0248853	&	0.0242091	&	0.0248765	&	0.0242189	&	0.0248669	\\
			&& 3 &	0.0242294	&	0.0248565	&	0.0242407	&	0.0248454	&	0.0242528	&	0.0248336	&	0.0242655	&	0.0248211	\\
			&& 4 &	0.024279	&	0.0248079	&	0.024293	&	0.0247941	&	0.0243077	&	0.0247797	&	0.024323	&	0.0247647	\\
			&& 5 &	0.0243388	&	0.0247492	&	0.0243551	&	0.0247332	&	0.0243718	&	0.0247167	&	0.0243889	&	0.0246999	\\
			&& 6 &	0.0244064	&	0.0246826	&	0.0244243	&	0.024665	&	0.0244424	&	0.0246471	&	0.0244607	&	0.0246289	\\
			&& 7 &	0.0244793	&	0.0246105	&	0.0244979	&	0.0245919	&	0.0245167	&	0.0245732	&	0.0245355	&	0.0245544	\\
			\bottomrule
		\end{tabular}
				}
\end{table}

The optimal angles $\alpha_k^*$ that produce $\geo{B}_n^*$ and $\geo{Q}_n^*$ are given in Tables~\ref{table:angles:Bn} and~\ref{table:angles:Qn}, respectively. We observe a pattern of damped oscillation, converging in an alterning manner to a mean value around $\pi/n$. For $\geo{Q}_n^*$, this observation leads to the following theorem:

\begin{theorem}\label{thm:Qn}
	Suppose $n=2^s$ with integer $s\ge 2$. Then there exists a convex small $n$-gon $\geo{Q}_n$ such that
	\[
	\begin{aligned}
		L(\geo{Q}_n) &= 2n \sin \frac{\pi}{2n} \cos \left( \frac{\pi}{8} - \frac{1}{2} \arcsin\left(\frac{1}{\sqrt{2}}\cos \frac{\pi}{n} \right) \right),\\
		W(\geo{Q}_n) &= \cos \left( \frac{\pi}{2n}  + \frac{\pi}{8} - \frac{1}{2} \arcsin\left(\frac{1}{\sqrt{2}}\cos \frac{\pi}{n} \right) \right),
	\end{aligned}
	\]
	and
	\[
	\begin{aligned}
		\ub{L}_n - L(\geo{Q}_n) &= \frac{\pi^5}{32n^4} + O\left(\frac{1}{n^6}\right),\\
		\ub{W}_n - W(\geo{Q}_n) &= \frac{\pi^3}{8n^3} + O\left(\frac{1}{n^4}\right).
	\end{aligned}
	\]
	In particular, for $n=4$, $L(\geo{Q}_4) = L_4^*$ and $W(\geo{Q}_4) = W_4^*$.
\end{theorem}
\begin{proof}
	The proof is similar to that of Theorem~\ref{thm:Bn}.
\end{proof}

For all $n=2^s$ with integer $s\ge 2$, the diameter graph of~$\geo{Q}_n$ has a cycle of length $n-1$ plus one pendant edge. We represent some polygons $\geo{Q}_n$ in Figure~\ref{figure:ngon:Qn}. They are all symmetrical with respect to the vertical pendant edge. In $\geo{Q}_n$, the angles $\alpha_0,\alpha_1,\ldots,\alpha_{\frac{n}{2}-1}$ defined in Figure~\ref{figure:model:Qn} are given by $\alpha_k = \pi/n - (-1)^k \gamma$ for all $k = 0,1,\ldots, n/2-1$, with
\[
\gamma = \frac{\pi}{4}-\arcsin\left(\frac{1}{\sqrt{2}}\cos \frac{\pi}{n}\right) = \frac{\pi^2}{2n^2} - \frac{\pi^4}{6n^4} + O\left(\frac{1}{n^6}\right).
\]
We can show that $L(\geo{Q}_n) > L(\geo{R}_{n-1}^+)$ and $W(\geo{Q}_n) < W(\geo{R}_{n-1}^+)$ when $n=2^s$ with integer $s\ge 3$.

\begin{figure}
	\centering
	\subfloat[$(\geo{Q}_4,3.0353,0.8660)$]{
		\begin{tikzpicture}[scale=4]
			\draw[dashed] (0.5000,0.8660) -- (0,1) -- (-0.5000,0.8660);
			\draw[blue,thick] (0,0) -- (0.5000,0.8660) -- (-0.5000,0.8660) -- cycle;
			\draw[red,thick] (0,0) -- (0,1);
		\end{tikzpicture}
	}
	\subfloat[$(\geo{Q}_{8},3.1193,0.9730)$]{
		\begin{tikzpicture}[scale=4]
			\draw[dashed] (0,0) -- (0.3933,0.2424) -- (0.5000,0.6919) -- (0.3138,0.9495) -- (0,1) -- (-0.3138,0.9495) -- (-0.5000,0.6919) -- (-0.3933,0.2424) -- cycle;
			\draw[blue,thick] (0,0) -- (0.3138,0.9495) -- (-0.3933,0.2424) -- (0.5000,0.6919) -- (-0.5000,0.6919) -- (0.3933,0.2424) -- (-0.3138,0.9495) -- cycle;
			\draw[red,thick] (0,0) -- (0,1);
		\end{tikzpicture}
	}
	\subfloat[$(\geo{Q}_{16},3.1364,0.9942)$]{
		\begin{tikzpicture}[scale=4]
			\draw[dashed] (0,0) -- (0.2063,0.0604) -- (0.3738,0.1952) -- (0.4769,0.3838) -- (0.5000,0.5976) -- (0.4470,0.7665) -- (0.3333,0.9023) -- (0.1764,0.9843) -- (0,1) -- (-0.1764,0.9843) -- (-0.3333,0.9023) -- (-0.4470,0.7665) -- (-0.5000,0.5976) -- (-0.4769,0.3838) -- (-0.3738,0.1952) -- (-0.2063,0.0604) -- cycle;
			\draw[blue,thick] (0,0) -- (0.1764,0.9843) -- (-0.2063,0.0604) -- (0.3333,0.9023) -- (-0.3738,0.1952) -- (0.4470,0.7665) -- (-0.4769,0.3838) -- (0.5000,0.5976) -- (-0.5000,0.5976) -- (0.4769,0.3838) -- (-0.4470,0.7665) -- (0.3738,0.1952) -- (-0.3333,0.9023) -- (0.2063,0.0604) -- (-0.1764,0.9843) -- cycle;
			\draw[red,thick] (0,0) -- (0,1);
		\end{tikzpicture}
	}
	\subfloat[$(\geo{Q}_{32},3.1403,0.9987)$]{
		\begin{tikzpicture}[scale=4]
			\draw[dashed] (0,0) -- (0.1019,0.0149) -- (0.1989,0.0493) -- (0.2873,0.1020) -- (0.3637,0.1710) -- (0.4252,0.2535) -- (0.4695,0.3464) -- (0.4947,0.4462) -- (0.5000,0.5490) -- (0.4861,0.6413) -- (0.4544,0.7291) -- (0.4063,0.8091) -- (0.3434,0.8781) -- (0.2683,0.9335) -- (0.1838,0.9732) -- (0.0932,0.9956) -- (0,1) -- (-0.0932,0.9956) -- (-0.1838,0.9732) -- (-0.2683,0.9335) -- (-0.3434,0.8781) -- (-0.4063,0.8091) -- (-0.4544,0.7291) -- (-0.4861,0.6413) -- (-0.5000,0.5490) -- (-0.4947,0.4462) -- (-0.4695,0.3464) -- (-0.4252,0.2535) -- (-0.3637,0.1710) -- (-0.2873,0.1020) -- (-0.1989,0.0493) -- (-0.1019,0.0149) -- cycle;
			\draw[blue,thick] (0,0) -- (0.0932,0.9956) -- (-0.1019,0.0149) -- (0.1838,0.9732) -- (-0.1989,0.0493) -- (0.2683,0.9335) -- (-0.2873,0.1020) -- (0.3434,0.8781) -- (-0.3637,0.1710) -- (0.4063,0.8091) -- (-0.4252,0.2535) -- (0.4544,0.7291) -- (-0.4695,0.3464) -- (0.4861,0.6413) -- (-0.4947,0.4462) -- (0.5000,0.5490) -- (-0.5000,0.5490) -- (0.4947,0.4462) -- (-0.4861,0.6413) -- (0.4695,0.3464) -- (-0.4544,0.7291) -- (0.4252,0.2535) -- (-0.4063,0.8091) -- (0.3637,0.1710) -- (-0.3434,0.8781) -- (0.2873,0.1020) -- (-0.2683,0.9335) -- (0.1989,0.0493) -- (-0.1838,0.9732) -- (0.1019,0.0149) -- (-0.0932,0.9956) -- cycle;
			\draw[red,thick] (0,0) -- (0,1);
		\end{tikzpicture}
	}
	\caption{Polygons $(\geo{Q}_n,L(\geo{Q}_n),W(\geo{Q}_n))$ defined in Theorem~\ref{thm:Qn}}
	\label{figure:ngon:Qn}
\end{figure}
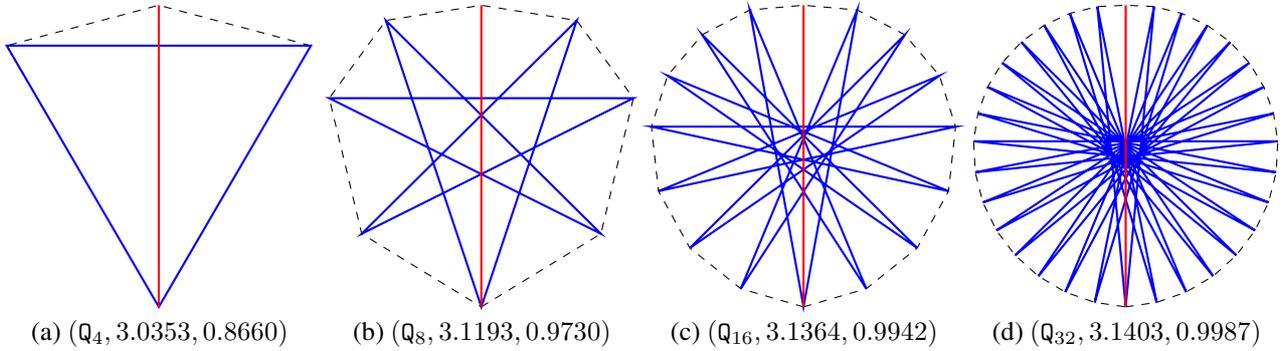

\section{Conclusion}\label{sec:conclusion}
Tighter lower bounds on the maximal perimeter and the maximal width of convex small $n$-gons were provided when $n$ is a power of $2$. For all $n=2^s$ with integer $s\ge 3$, we constructed a convex small $n$-gon $\geo{B}_n$ whose perimeter and width cannot be improved for large $n$ by more than $\frac{\pi^7}{32n^6}$ and $\frac{\pi^4}{8n^4}$ respectively.

In addition, under the assumption that Mossinghoff's conjecture is true, we formulated the maximal perimeter problem as a nonlinear optimization problem involving trigonometric functions and provided global optimal $n$-gons for $n=2^s$ and $3\le s \le 7$.

\section*{Acknowledgements}
The author thanks Charles Audet, Professor at Polytechnique Montr\'{e}al, for helpful discussions on extremal small polygons and helpful comments on early drafts of this paper.

\bibliographystyle{ieeetr}
\bibliography{../../research}

\end{document}